\documentclass[10pt, leqno]{amsart}

\usepackage{fullpage}
\usepackage{amsmath,amssymb,amsthm, epsfig}

\usepackage{hyperref}

\usepackage{amsmath}

\usepackage[pdftex]{color}

\usepackage{color}

\usepackage{ulem}

\usepackage{epsfig}

\usepackage{mathrsfs}

\usepackage{wasysym}

\usepackage{stackrel}

\usepackage{ifthen}

\usepackage{color, xcolor}

\def \N {\mathbb{N}}
\def \R {\mathbb{R}}

\def \tr {\mathrm{Tr}}
\def \esslimsup{\mathrm{ess limsup}}
\def \essliminf{\mathrm{ess liminf}}

\def \diam {\mathrm{diam}}

\newcommand{\defeq}{\mathrel{\mathop:}=}

\newtheorem{theorem}{Theorem}[section]

\newtheorem{lemma}[theorem]{Lemma}

\newtheorem{proposition}[theorem]{Proposition}

\newtheorem{corollary}[theorem]{Corollary}

\theoremstyle{definition}

\newtheorem{definition}[theorem]{Definition}

\newtheorem{Assumption}{A}

\theoremstyle{remark}

\newtheorem{remark}[theorem]{Remark}

\numberwithin{equation}{section}

\allowdisplaybreaks

\newcommand{\intav}[1]{\mathchoice {\mathop{\vrule width 6pt height 3 pt depth  -2.5pt
\kern -8pt \intop}\nolimits_{\kern -6pt#1}} {\mathop{\vrule width
5pt height 3  pt depth -2.6pt \kern -6pt \intop}\nolimits_{#1}}
{\mathop{\vrule width 5pt height 3 pt depth -2.6pt \kern -6pt
\intop}\nolimits_{#1}} {\mathop{\vrule width 5pt height 3 pt depth
-2.6pt \kern -6pt \intop}\nolimits_{#1}}}

\begin{document}

\title{Schauder and Calder\'{o}n-Zygmund type estimates for fully nonlinear parabolic equations under ``small ellipticity aperture'' and applications}

\author{Jo\~{a}o Vitor  da Silva and Makson S. Santos}

\maketitle

\begin{abstract}

In this manuscript we derive Schauder estimates for viscosity solutions to non-convex fully nonlinear second order parabolic equations
\[
\partial_t u - F(x, t,D^2u) = f (x, t) \quad \text{in} \quad \mathrm{Q}_1 = B_1 \times (-1, 0],
\]
provided that the source $f$ and the coefficients of $F$ are H\"{o}lder continuous functions and $F$ enjoys a small ellipticity aperture. Furthermore, for problems with merely bounded data, we prove that such solutions are $C^{1, \text{Log-Lip}}$ smooth in the parabolic metric. We also address Calder\'{o}n-Zygmund estimates for such class of non-convex operators. Finally, we connect our findings with recent estimates for fully nonlinear models in certain solutions classes.

\medskip


\vspace{0.2cm}

\end{abstract}

\tableofcontents


\section{Introduction}

In this article we investigate regularity estimates to viscosity solutions of non-convex fully nonlinear equations of the form
\begin{equation}\label{eq_main}
	\partial_t u - F(x, t, D^2 u) = f(x,t) \quad \textrm{in} \quad Q_1,
\end{equation}
under appropriate assumptions on $f$, the coefficients,  and the ellipticity aperture of the operator $F$, defined as $\mathfrak{e} := \frac{\Lambda}{\lambda} -1$. We produce $C^{2+\alpha,\frac{2+\alpha}{2}}$, parabolic-$C^{1, Log-Lip}$ and $W^{2,1, p}$-estimates, each tailored to different scenarios related to the regularity of $f$.

It has long been established, through Krylov-Safonov's Harnack inequality (as documented in \cite{KrySaf80}, \cite{CKS00}, and \cite{Wang92-I}), that viscosity solutions to constant coefficient equations of the form
\begin{equation}\label{EqHom}
\partial_t u - F(D^2 u) = 0 \;\; \mbox{ in } \;\; Q_1,
\end{equation}
exhibit local differentiability. Moreover, there exists a universal constant $\alpha \in (0, 1)$ such that
\[
\|u\|_{C^{1+\alpha, \frac{1+\alpha}{2}}(Q_{1/2})} \leq \mathrm{C}(\verb"universal")\|u\|_{L^{\infty}(Q_1)}.
\]

In the pursuit of classical solutions (i.e., $C^{2,1}$ solutions) to fully nonlinear equations like \eqref{EqHom}, ground-breaking contributions emerged from the works of Evans \cite{Evans82} and Krylov \cite{Krylov83}. It states that under concavity or convexity assumption on $F$, solutions to \eqref{EqHom} are $C^{2 +\alpha, \frac{2 +\alpha}{2}}(Q_1)$, for some universal $\alpha \in (0, 1)$  and we have
\[
\|u\|_{C^{2+\alpha, \frac{2+\alpha}{2}}(Q_{1/2})} \leq \mathrm{C}(\verb"universal")\|u\|_{L^{\infty}(Q_1)}.
\]

It is worth mentioning that Schauder estimates for 2nd order fully nonlinear (convex or concave) operators with Dini continuous data as  \eqref{eq_main} have been considerably well studied in the last decades, see for instance \cite[Section 1.1]{Wang92-II} for the H\"{o}lder continuous setting and, \cite[Theorem 1]{DasdosP19} and \cite{TW13} for similar results in a more general scenario.

The question of whether any fully nonlinear parabolic operator could enjoy an a priori $C^{2, 1}$ regularity theory has intrigued the mathematical community for several decades. Notably, Silvestre's work in \cite{Sil22} presented a compelling example of a non-homogeneous solution to a uniformly parabolic equation that exhibited an isolated singularity and failed to meet the $C^{2, 1}$ criteria. Similarly, Caffarelli-Stefanelli's counterexample in \cite{CafSte08} shed light on the fact that $C^{2,1}$-regularity is generally not guaranteed. Consequently, classical solutions are primarily secured under the assumption of convexity or concavity on $F$.

These counterexamples by Caffarelli-Stefanelli and Silvestre partially close the case for $C^{2, 1}$ regularity. Nevertheless, on the other hand, they simultaneously open up an even broader room for investigation. In effect, in view of the impossibility of a general existence theory for classical solutions to fully nonlinear equations, it becomes a central theme of research to obtain additional conditions on $F$ and on $u$ in order to establish $C^{2, 1}$ estimates.

On this subject, here we paraphrase Cabr\'{e} and Caffarelli in \cite[page 2]{CabCaf2003} for our context:

\bigbreak

\begin{flushright}
\begin{minipage}{8cm}
``Which assumptions on $F$, in between convexity of $F$ and no assumptions, and perhaps depending on the dimension $d$, guarantee that solutions of \eqref{eq_main} are classical?''
\end{minipage}
\end{flushright}

\bigbreak

With this question in mind, we state the first main  result of the paper:

\begin{theorem}[{\bf $C_{\text{loc}}^{2 +\alpha, \frac{2 +\alpha}{2}}$  estimates}] \label{theo:main1} Let $u\in C(Q_1) $ be a bounded viscosity solution to \eqref{eq_main} such that A\ref{Assump1}-A\ref{A3} hold true. Given $\alpha \in (0, 1)$, there exists $ \varepsilon_0 > 0$, depending only upon $d$ and $\alpha$, such that if
\[
\frac{\Lambda}{\lambda} -1 \leq \varepsilon_0,
\]
then $u\in C_{\text{loc}}^{2 +\alpha, \frac{2 +\alpha}{2}}(Q_{{1}/{2}})$ and there exists a positive constant $\mathrm{C}=\mathrm{C}(d, \lambda, \Lambda, \alpha)$ such that
\begin{equation}\label{EstSchauder}
      \|u\|_{C^{2 +\alpha, \frac{2 +\alpha}{2}}(Q_{1/2})}  \leq \mathrm{C}(\|u\|_{L^\infty(Q_1)}+\|f\|_{C^{\alpha, \frac{\alpha}{2}}(Q_1)}).
\end{equation}
\end{theorem}

The estimates addressed by Theorem \ref{theo:main1} must be understood as the parabolic counterpart of the ones established in \cite{WuNiu23} in the elliptic scenario.

In the lack of a universal $\alpha$-H\"{o}lder modulus of continuity for the source term, we may produce asymptotic estimates in borderline Log-Lipschitz spaces. This result is enclosed in previous works, \cite{DasdosP19} and \cite{DaST17}, in the framework of flat solutions or convex operators.

\begin{theorem}[{ \bf Parabolic $C^{1, \text{Log-Lip}}$ estimates}]\label{theo:LL} Let $u\in C(Q_1) $ be a bounded viscosity solution to \eqref{eq_main} such that A\ref{Assump1}, A\ref{A2} and A\ref{ALL} hold true. There exists $ \varepsilon_0 > 0$, depending only upon $d,\alpha$, such that if
\[
\frac{\Lambda}{\lambda} -1 \leq \varepsilon_0
\]
then $u\in C^{1, \mathrm{Log-Lip}}(Q_{1/2})$ and
\[
  \|u\|_{C^{1, Log-Lip}(Q_{1/2})} \leq \mathrm{C}(\|u\|_{L^\infty(Q_1)}+\|f\|_{L^{\infty}(Q_1)}),
\]
where $C=\mathrm{C}(d, \lambda, \Lambda)$ is a positive constant.
\end{theorem}

Additionally, we will provide $W_{\text{loc}}^{2,1,  p}$ estimates for operators under small ellipticity aperture. The fully nonlinear parabolic theory in Sobolev spaces (or Calder\'{o}n-Zygmund estimates) is quite recent and less is known when the PDE involves discontinuous ingredients and solutions are investigated from the Sobolev's point of view. In such a context, we highlight the Crandall's \textit{et al} results for convex operators in \cite[Theorem 2.8]{CKS00} (cf. \cite{Wang92-I} for related results). More precisely, viscosity solutions to
\[
\partial_t u - F(x, t, D^2 u) = f(x,t) \;\;\mbox{ in } \;\;Q_1,
\]
with $f \in L^p(Q_1)$, for $p>d+1$, satisfy the following regularity estimates
\[
 \|u\|_{W^{2,1,p}(Q_{1/2})} \leq \mathrm{C}(\verb"universal")\left(\|u\|_{L^\infty(Q_1)}+\|f\|_{L^p(Q_1)}\right).
\]
We also refer the reader to the works of \cite{DongKryl13}, \cite{DongKryl19} and \cite{ZZZ21} on Hessian type estimates for fully nonlinear parabolic PDEs with measurable coefficients, weighted and mixed-norm Sobolev spaces and/or under a sort of relaxed convexity assumptions (including oblique boundary conditions).

Therefore, we can ask a similar question as before: Is it possible to establish $W^{2,1, p}$ estimates, without imposing convexity or concavity on the nonlinearity $F$? In this scenario, the authors in \cite{CastPim17} proved Hessian estimates by using asymptotic assumptions on the operator $F$, with the notion of recession function.

Our next main theorem states that Calder\'{o}n-Zygmund estimates for non-convex fully nonlinear equations also hold under the assumption of ``small ellipticity aperture".

\begin{theorem}[{\bf $W_{{loc}}^{2, 1, p}$ parabolic estimates}] \label{theo:main3} Let $u\in C(Q_1) $ be a bounded viscosity solution to \eqref{eq_main} such that A\ref{Assump1}, A\ref{A2} and A\ref{A4} are in force. There exists $ \varepsilon_0 > 0$, such that if
\[
\frac{\Lambda}{\lambda} -1 \leq \varepsilon_0
\]
then $u \in W_{loc}^{2, 1, p}(Q_{ {1}/{2}})$ and there exists a positive constant $\mathrm{C}=\mathrm{C}(d, \lambda, \Lambda, p)$ such that
\[
    \|u\|_{W^{2,1,p}(Q_{1/2})} \leq \mathrm{C}(\|u\|_{L^\infty(Q_1)}+\|f\|_{L^p(Q_1)}).
\]
\end{theorem}

In conclusion, our findings can play a crucial role in the study of regularity for other nonlinear parabolic models (cf. \cite{ARS22} and  \cite{AndSant22}), as well as in their applications to certain free boundary problems of obstacle-type (cf. \cite{AudKuk23} and \cite{LindMon13}).

\section{Preliminaries and main assumptions} \label{DPR}

\subsection{Notation}
In this section, we present the notation that will be used in this manuscript. By $d+1$, we denote the dimension of the Euclidean space $\R^{d+1}$. As usual, the {\it open ball} of radius $r$ and center at $x_0$ stands for $B_r(x_0)$. We denote the {\it parabolic cylinder} as
\[
Q_r(x_0, t_0) = B_r(x_0) \times (t_0 -r^2, t_0].
\]
When $(x_0,t_0) = (0,0)$,  we simply write $Q_r$. The {\it parabolic cube} of side $r$ is defined by
\[
K_r := [-r,r]^d\times[-r^2,0].
\]

Given $(x_1, t_1), (x_2, t_2) \in Q_r$ we define the {\it parabolic distance between them as
\[
dist((x_1, t_1), (x_2, t_2)) := |x_1-x_2| + \sqrt{|t_1-t_2|}.
\]

We say that $u : Q_r \rightarrow \R$ belongs to the {\it parabolic H\"older space} $C^{\alpha, \frac{\alpha}{2}}(Q_r)$ if
\[
\|u\|_{C^{\alpha, \frac{\alpha}{2}}(Q_r)} :=  \|u\|_{L^\infty(Q_r)} + [u]_{C^{\alpha, \frac{\alpha}{2}}(Q_r)} < \infty,
\]
where the semi-norm $[u]_{{\mathcal C}^{\alpha, \frac{\alpha}{2}}(Q_r)}$ is given by
\[
[u]_{C^{\alpha, \frac{\alpha}{2}}(Q_r)} := \sup_{(x, t),(y, s)\in Q_r \atop{(x, t) \not= (y, s)}} \frac{|u(x,t)-u(y, s)|}{\mbox{dist}((x, t),(y, s))^{\alpha}}.
\]

In particular, we have that $u$ is $\alpha$-H\"older continuous with respect to the spatial variables and $\frac{\alpha}{2}$-H\"older continuous with respect to the time variable.

Similarly, a function $u \in C^{1+\alpha, \frac{1+\alpha}{2}}(Q_r)$ if its spatial gradient $Du(x,t)$ exists in the classical sense and

\begin{align*}
\|u\|_{C^{1+ \alpha, \frac{1+\alpha}{2}}(Q_r)} := \;\; & \|u\|_{L^{\infty}(Q_r)} + \|Du\|_{L^{\infty}(Q_r)}
 + [Du]_{C^{\alpha, \frac{\alpha}{2}}(Q_r)}  < \infty
\end{align*}
where
$$
[Du]_{C^{\alpha, \frac{\alpha}{2}}(Q_r)}: = \sup_{(x, t),(y, s)\in Q_r \atop{(x, t) \not= (y, s), |\kappa|=1}} \frac{|D^{\kappa}u(x,t)-D^{\kappa}u(y, s)|}{dist((x, t),(y, s))^{\alpha}}.
$$

As before, we have that $Du$ is $\alpha$-H\"older continuous and $u$ is $\frac{1+\alpha}{2}$-H\"{o}lder continuous in the time variable. For more details, see \cite[Section 1]{CKS00}.

Finally, a function belongs to $C^{2+\alpha, \frac{2+\alpha}{2}}(Q_1)$ if its spatial Hessian $D^2u(x,t)$ and time derivative exist in the classical sense, and the norm

\begin{align*}
\|u\|_{C^{2+ \alpha, \frac{2+\alpha}{2}}(Q_r)} := \;\; & \|u\|_{L^{\infty}(Q_r)} +\|\partial_t u\|_{L^{\infty}(Q_r)} + \|Du\|_{L^{\infty}(Q_r)} + \|D^2u\|_{L^{\infty}(Q_r)} +  [u]_{C^{2+\alpha, \frac{2+\alpha}{2}}(\mathrm{Q}_r)}
\end{align*}
is finite, where
\[
[u]_{C^{2+\alpha, \frac{2+\alpha}{2}}({Q_r})} = [\partial_t u]_{C^{\alpha, \frac{\alpha}{2}}({Q_r})} + [D^2 u]_{C^{\alpha, \frac{\alpha}{2}}({Q_r})}.
\]

Again, we have that every component of $D^2u$ is $\alpha$-H\"older continuous, and the time derivative $\partial_t u$ is $\frac{\alpha}{2}$-H\"older continuous.

In the next two results we present alternatives ways to control the norm $\|u\|_{C^{2+\alpha, \frac{2+\alpha}{2}}(\mathrm{Q})}$, in terms of quantities like $[u]_{C^{2+\alpha, \frac{2+\alpha}{2}}(\mathrm{Q})}$ and $\|u\|_{L^{\infty}(\mathrm{Q})}$. We refer the reader to \cite{Krylov96}.

\begin{lemma}[{\bf Interpolation inequalities}]\label{Lem01} For all $\epsilon_0 > 0$, there exists $\mathrm{C}(\epsilon_0)>0$ such that for all $u \in C^{2+\alpha, \frac{2+\alpha}{2}}(\mathrm{Q})$
\[
\left\{
\begin{array}{r}
\displaystyle  \left\|\partial_t u\right\|_{L^{\infty}(\mathrm{Q})}  \leq  \epsilon_0[u]_{C^{2+\alpha, \frac{2+\alpha}{2}}(\mathrm{Q})} + \mathrm{C}(\epsilon_0)\|u\|_{L^{\infty}(\mathrm{Q})} \\
 \displaystyle \|D u\|_{L^{\infty}(\mathrm{Q})} \leq  \epsilon_0[u]_{C^{2+\alpha, \frac{2+\alpha}{2}}(\mathrm{Q})} + \mathrm{C}(\epsilon_0)\|u\|_{L^{\infty}(\mathrm{Q})} \\
 \displaystyle  \|D^2 u\|_{L^{\infty}(\mathrm{Q})}  \leq  \epsilon_0[u]_{C^{2+\alpha, \frac{2+\alpha}{2}}(\mathrm{Q})} + \mathrm{C}(\epsilon_0)\|u\|_{L^{\infty}(\mathrm{Q})}
\end{array}
\right.
\]
\end{lemma}

\begin{lemma}[{\bf Equivalent semi-norms}]\label{Lem02} There exists $C \geq 1$ such that for
all $u \in C^{2+\alpha, \frac{2+\alpha}{2}}(\mathrm{Q})$ we have
\[
C^{-1}[u]^{\prime}_{C^{2+\alpha, \frac{2+\alpha}{2}}(\mathrm{Q})} \leq [u]_{C^{2+\alpha, \frac{2+\alpha}{2}}(\mathrm{Q})} \leq C [u]^{\prime}_{C^{2+\alpha, \frac{2+\alpha}{2}}(\mathrm{Q})}
\]
where
\[
[u]^{\prime}_{C^{2+\alpha, \frac{2+\alpha}{2}}(\mathrm{Q})} := \sup_{(x, t)  \in \mathrm{Q}} \sup_{\rho>0} \inf_{P \in \mathcal{P}} \frac{\|u-P\|_{L^{\infty}(Q_{\rho}\cap \mathrm{Q})}}{\rho^{2+\alpha}}
\]
and
\[
\mathcal{P} := \left\{\frac{1}{2}x^{T}Dx + Ct + B \cdot x +A : A\,, C \in \mathbb{R},\,B \in \mathbb{R}^d \;\text{ and }\; D \in \mbox{Sym}(d)\right\}.
\]
\end{lemma}

Now, the parabolic space $C^{1, \text{Log-Lip}}(Q_r)$ consists of functions $u \in C(Q_r)$ such that the quantity

\[
\displaystyle [u]_{C^{1, \text{Log-Lip}}(Q_r(x_0, t_0))} \defeq   \sup_{(x, t)  \in Q_r(x_0, t_0)} \sup_{\rho>0} \inf_{L \in \mathcal{L}} \frac{|u(x, t)-L(x, t)|}{r^2\log r^{-1}}
\]
is finite, where
$$
\mathcal{L} := \left\{B \cdot x +A : A \in \mathbb{R} \quad \text{and} \quad B \in \mathbb{R}^d \right\}.
$$
Moreover, the $C^{1, \text{Log-Lip}}$-norm is defined by
\[
\|u\|_{C^{1, \text{Log-Lip}}(Q_r(x_0, t_0))} \defeq \|u\|_{L^{\infty}(Q_r(x_0, t_0))} + \|Du\|_{L^{\infty}(Q_r(x_0, t_0))}+ [u]_{C^{1, \text{Log-Lip}}(Q_r(x_0, t_0))}.
\]

Given $p \in [1,\infty]$, the parabolic Sobolev space $W^{2,1,p}(Q_r)$ consists of functions $u \in L^p(Q_r)$ that satisfies $D u, D^2 u, u_t \in L^p(Q_1)$, i.e.,
\[
W^{2,1,p}(Q_1) := \{ u \in L^p(Q_r)\;|\; Du, D^2 u, u_t \in L^p(Q_r)\}.
\]
The corresponding norm is given by
\[
\|u\|_{W^{2,1,p}(Q_r)} := \left[\|u\|^p_{L^p{(Q_r)}}+\|u_t\|_{L^{p}(Q_r)}^p+\|D u\|_{L^{p}(Q_r)}^p+ \|D^2 u\|_{L^{p}(Q_r)}^p\right]^{\frac{1}{p}}.
\]

\subsection{Definitions and auxiliary results}

We define the Pucci's extremal operators $\mathcal{M}^{\pm}_{\lambda,\Lambda}$ as
\[
\mathcal{M}^{+}_{\lambda,\Lambda}(M) \defeq \lambda \cdot \sum_{e_i <0} e_i  + \Lambda \cdot \sum_{e_i >0} e_i \;\; \mbox{ and } \;\; \mathcal{M}^{-}_{\lambda,\Lambda}(M) \defeq \lambda \cdot \sum_{e_i >0} e_i  + \Lambda \cdot \sum_{e_i <0} e_i,
\]
where $\{e_i : 1 \le i \le d\}$ denote the eigenvalues of $M$. An operator $F : \mathrm{Q}_1 \times Sym(d) \to \R$ is called uniformly elliptic if there exist constants $0 < \lambda \leq \Lambda$ (ellipticity constants) such that for any $M, N \in\text{Sym}(d) $, with $N\ge 0$ and all $(x, t) \in Q_1$, we have
\begin{equation} \label{eq_ellip}
   \lambda\|N\|\leq F(x, t, M+N)-F(x, t, M)\leq\ \Lambda\|N\|.
\end{equation}
Equivalently, we also have for all $M,N \in Sym(d)$ and $(x,t) \in Q_1$
\[
\mathcal{M}^-(M-N) \leq F(x, t, M)-F(x, t, N) \leq \mathcal{M}^+(M-N).
\]
We say that a constant is {\it universal} if it only depends on $\lambda$, $\Lambda$ and $d$.

In what follows, we introduce the viscosity solutions, the appropriate notion of weak solutions for fully nonlinear elliptic equations. We refer the reader to \cite{CKS00} and \cite{Wang92-I}.

\begin{definition}[{\bf $L^{p}-$viscosity solutions}]\label{VS} Let $F: Q_1 \times Sym(d)\to \mathbb{R}$ be a uniformly elliptic operator, $p> \frac{d+2}{2}$ and $f \in L_{loc}^{p}(Q_1)$. We say that a function $u \in C(Q_1)$ is an $L^{p}-$viscosity sub-solution (respectively super-solution) to

\begin{equation}\label{eqVS}
   \partial_t u(x, t) - F\left(x,t, D^2 u(x,t)\right) = f(x,t) \quad in \quad Q_1,
\end{equation}
if for all $\varphi \in W_{loc}^{2, 1 , p}(Q_1)$, whenever $\varepsilon >0$ and $\mathcal{O} \subset Q_1$ is an open set and
$$
\partial_t \varphi(x, t)-F\left(x, t, D^2 \varphi(x, t)\right) - f(x, t) \geq  \varepsilon \quad \left(resp. \leq - \varepsilon\right)\quad a.e. \quad in \quad \mathcal{O},
$$
then $u-\varphi$ cannot attain a local maximum (resp. minimum) in $\mathcal{O}$. In an equivalent manner, $u$ is an $L^{p}-$viscosity sub-solution (resp. super-solution) if for all test function $\varphi \in W_{loc}^{1, 2 , p} (Q_1)$ and $(x_0, t_0) \in Q_1$ at which $u - \varphi$ attains a local maximum (resp. minimum) one has

\begin{equation}\label{DefVisSol}
\left\{\begin{array}{c}
\displaystyle \stackrel[(x, t) \to (x_0, t_0)]{}{\essliminf} \left[\partial_t \varphi(x, t)-F\left(x, t, D^2 \varphi(x, t)\right) - f(x, t)\right] \leq 0\\
    \displaystyle \stackrel[(x, t) \to (x_0, t_0)]{}{\esslimsup}  \left[\partial_t \varphi(x, t)-F\left(x, t, D^2 \varphi(x, t)\right) - f(x, t)\right] \geq 0
\end{array}
\right.
\end{equation}

Finally, we say that $u$ is an $L^{p}-$viscosity solution to \eqref{eqVS} if it is both an $L^{p}-$viscosity super-solution and an $L^{p}-$viscosity sub-solution.
\end{definition}

\begin{remark}
We say that a function $u \in C(Q_1)$ is a $C$-viscosity solution to \eqref{eqVS} when the sentences in \eqref{DefVisSol} are evaluated point-wisely for all ``test functions'' $\varphi \in C_{loc}^{2, 1}(Q_1)$. This is the notion used by Imbert-Silvestre in \cite{ImbSilv13} and Wang in \cite{Wang92-I} and \cite{Wang92-II}.
\end{remark}

\begin{definition}[{\bf The class of viscosity solutions}]
Let $f\in{C}(Q_1)$ and $0<\lambda\leq\Lambda.$ We say that $u \in C(Q_1)$ belongs to the class $\underline{S}(\lambda, \Lambda, f)$ if
\[
	\partial_t u -\mathcal{M}^+_{\lambda,\Lambda}(D^2u)\,\leq\, f(x,t) \;\; \mbox{ in } \;\; Q_1
\]
in the viscosity sense. Similarly, $u \in C(Q_1)$ belongs to the class $\overline{S}(\lambda, \Lambda, f)$ if
\[
	\partial_tu-\mathcal{M}^-_{\lambda,\Lambda}(D^2u)\,\geq\, f(x,t) \,\, \mbox{ in } \,\, Q_1
\]
	in the viscosity sense. Finally, the class of $(\lambda, \Lambda)$-viscosity solutions is defined by
\[
	S(\lambda,\Lambda,f)=\overline{S}(\lambda, \Lambda, f)\cap\underline{S}(\lambda, \Lambda, f).
\]
\end{definition}

For a fixed $(x_0, t_0) \in Q_1$, the oscillation of the coefficients of $F$ around $(x_0, t_0)$ is defined by
\begin{equation}\label{eq_osc}
\displaystyle \Theta_F(x_0, t_0, x, t) := \sup_{M \in Sym(n)}
\frac{|F(x, t, M)-F(x_0, t_0, M)|}{\|M\|+1}.
\end{equation}
By simplicity, we denote $\Theta_F(0, 0, x, t) = \Theta_F(x, t)$.

In the next, we present some measure notions that we use in this manuscript. We refer the reader to \cite{Wang92-II} for more details.

\begin{definition}
Given a affine function	$L:Q_1\rightarrow\mathbb{R}$ and a positive constant $M$, the paraboloid of opening $M$ is denoted by
\[
	P_{M}(x,t) = L(x,t) \pm M(|x|^2 + |t|).
\]
Moreover, we introduce the sets	
\[
	\underline{G}_{M}(u, Q) := \{(x_0,t_0) \in Q: \exists \, P_{M} \;\mbox{ that touches } u \mbox{ from bellow at} \; (x_0,t_0)\},
\]
\[
	\overline{G}_{M}(u, Q):\{(x_0,t_0)\in Q: \exists \, P_{M} \; \mbox{ that touches } u \mbox{ from above at} \; (x_0,t_0)\},
\]
and
\[
	G_M(u,Q):= \underline{G}_{M}(u, Q)\cap \overline{G}_{M}(u, Q).
\]
Finally, we denote
\[
	\underline{A}_M(u,Q):= Q\setminus\underline{G}_M(u,Q),\;\;\;\; \bar{A}_M(u,Q) := Q\setminus\overline{G}_M(u,Q)
\]
	and
\[
	A_M(u,Q):= \underline{A}_{M}(u, Q)\cup \bar{A}_{M}(u, Q).
\]
\end{definition}

Starting with $K_1$, a dyadic cube is constructed by iteratively following this procedure: We divide the sides of $K_1$ into two equal segments in the x-axis and four equal segments in the t-axis. This division is also applied to the $2^{d+2}$ resulting cubes, and the process is repeated. Each cube generated in this manner is referred to as a dyadic cube. We refer to $\tilde{K}$ as a precursor to a cube $K$, if $K$ is one of the $2^{d+2}$ cubes created by subdividing the sides of $\tilde{K}$.

In addition, given $m \in \mathbb{N}$ and a dyadic cube $K$, the set $\bar{K}^m$ is obtained by staking $m$ copies of its predecessor $\bar{K}$; in other words, if $\bar{K}$ has the form $(a,b)\times L$, then $\bar{K}^m = (b,b + m(b-a))\times L$. With this notation we have the following result (see \cite[Lemma 4.27]{ImbSilv13} and \cite{Wang92-II}):

\begin{lemma}[{ \bf Stacked covering lemma}]\label{lem_cov}
Let $m \in \mathbb{N}$, $A \subset B \subset K_1$ and $0<\rho <1$. Suppose that
\begin{itemize}
\item[(i)] $|A| \leq \rho|K_1|$;
\item[(ii)] If $K$ is dyadic cube of $K_1$ such that $|K \cap A| > \rho|K|$, then $\bar{K}^m\subset B$.
\end{itemize}
Then $|A| \leq \dfrac{\rho(m+1)}{m} |B|$.

Here $|\cdot|$ stands for the $(d+1)$-Lebesgue measure.
\end{lemma}

We close this section with some auxiliary results regarding the norms of a solution $u$. We start with a parabolic A.B.P.K.T.-estimate, which ensures that we can consider bounded viscosity solutions. See \cite{Kryl76, Krylov76,Krylov87,Tso85} for more details.

\begin{lemma}[{\bf Aleksandrov-Bakelman-Pucci-Krylov-Tso Maximum Principle}]
Let $u \in W_{\text{loc}}^{2, 1, d+1}(\mathrm{Q}) \cap C(\bar{Q})$ be a strong subsolution of
\[
\partial u_t - \mathcal{M}^{-}_{\lambda,\Lambda}(D^2 u) -  \gamma|Du| \le f(t, x) \quad \text{in} \quad \mathrm{Q},
\]
where, $Q = (-T, 0] \times \Omega$, $\gamma \geq  0$ and $f \in L^{d+1}(\mathrm{Q})$. Then, there exists a positive constant $C = C(d, \lambda, \Lambda, \gamma, \mbox{diam}(\Omega))$, such that
\[
 \displaystyle \sup_{\mathrm{Q}} u(x, t) \le \sup_{\partial_p \mathrm{Q}} u(x, t) + \mathrm{C}\cdot (\diam(\Omega))^{\frac{d}{d+1}} \|f\|_{L^{d+1}(\Gamma^+_{\mathrm{Q}})}.
\]
where $\Gamma^+_{\mathrm{Q}}$ is the parabolic upper contact set of $u$.
\end{lemma}

Finally, we recall interior derivative estimates for caloric equations. See \cite[Chap. 8, p. 116]{Krylov96}.

\begin{lemma}[{\bf Interior derivative estimates}]\label{EstIntDer} There exists a positive constant $C$ such that any solution of
\[
\partial_t h = \Delta h \;\;\mbox{ in }\;\;\mathrm{Q}_R
\]
satisfies
\[
\left|\partial^k_t D^{\beta} h(0, 0)\right| \leq \mathrm{C}\frac{\|h\|_{L^{\infty}(Q_R)}}{R^{2k+|\beta|}},
\]
where $\beta = (\beta_1, \cdots, \beta_n)$ is a multi-index, $\displaystyle |\beta| = \sum_{i=1}^{n} \beta_i$ and $D^{\beta} h = \partial_{x_1}^{\beta_1} \ldots \partial_{x_n}^{\beta_n}h$.
\end{lemma}

\subsection{Main assumptions}\label{Assump}

In this subsection, we detail the main assumptions used throughout the paper. We start with an assumption on the operator $F$.

\begin{Assumption}\label{Assump1}[{\bf Uniform Ellipticity}] We suppose that the operator $F: \mathrm{Q}_1 \times Sym(d) \to \R$ is a uniformly $(\lambda, \Lambda)$-elliptic operator.
\end{Assumption}
\begin{Assumption}\label{A2}[{\bf Reducibility condition}] We will assume that
\begin{equation} \label{eq_normalization}
    F(0,0, 0_{d\times d}) =  f(0,0) = 0.
\end{equation}
\end{Assumption}
Notice that Assumption A\ref{A2}, is not restrictive. In fact, we can always define $G(x,t,M) := F(x,t,M) - F(0, 0, 0_{n\times n})$, so that $G$ satisfies A\ref{A2} (and A\ref{Assump1}). Similarly, we can define $g(x,t) := f(x,t)-f(0,0)$. In the next, we impose some regularity conditions on the operator $F$ and the source term $f$.
\begin{Assumption}\label{A3}[{\bf Parabolic H\"{o}lder condition}] For a modulus of continuity $\tau$, there holds
\begin{equation}\label{cont coeff}
    \frac{|F(x, t, M)-F(y, s, M)|}{\|M\|+1} \leq  \tau(dist((x, t),(y, s))), \quad \mbox{ for } M \in \mbox{Sym}(d).
\end{equation}
\begin{equation}\label{cont source}
     |f(x, t)- f(y, s)| \leq   \tau(dist((x, t),(y, s))), \quad    \mbox{ for } (x, t),(y, s)\in Q_{1}.
\end{equation}
where $\tau(s) \leq \mathrm{C}_0s^{\alpha}$ for some $\alpha \in (0, 1)$.
\end{Assumption}
Notice that such conditions are essential to surpass the $C^{2,1}$-threshold. See for instance \cite{CZ02}, \cite{TW13} for surveys on this topic in the fully nonlinear framework. Next, we present the assumption we impose on $f$ to obtain the Log-Lipschitz estimates.

\begin{Assumption}\label{ALL}[{\bf Boundedness of the source term}\,] We suppose that $f \in L^\infty(Q_1)$.
\end{Assumption}
Finally, we give the necessary conditions on $f$ to obtain the Sobolev regularity.

\begin{Assumption}\label{A4}[{\bf Integrability of the source term}\,] We suppose that $f \in L^p(Q_1)$, for $p>d+1$.
\end{Assumption}

\medskip

The remainder of this paper is structured as follows: In Section 3, we derive the Schauder estimates. Section 4 is devoted to the proof of the Log-Lipschitz regularity. In Section 5, we give the proof of the regularity in Sobolev spaces. We close the paper with some applications of our results in Section 6.

\section{Schauder estimates}\label{sect_scha}

This section is devoted to the prove of Theorem \ref{theo:main1}. We start with a crucial lemma, which states that our solutions can be approximated by a bounded second order polynomial.

\smallbreak

\begin{lemma}[{\bf Approximation Lemma}]\label{lemma_app} Let $u \in C(Q_1)$ be a normalized viscosity solution to \eqref{eq_main}. Suppose that A\ref{Assump1} and A\ref{A2} are in force. There exists $\varepsilon > 0$ such that if
\[
   \Theta_F(x, t) + \|f\|_{L^{\infty }(Q_{1})} \leq \varepsilon, \;\; \mbox{ and } \;\; \Lambda \leq (1+\varepsilon)\lambda,	\\\\\\
\]
then we can find $0 < \rho \ll 1$ and a second order polynomial $P$ of the form
\[
P(x,t) := \dfrac{1}{2}x^TDx + Ct + B\cdot x + A,
\]
with $|A|, |B|, |C|, |D| \leq M$, that solves
\begin{equation}\label{eq_poly}
\partial_t P - \lambda\tr(D^2P) = 0 \;\; \mbox{ in } \; Q_{4/5}.
\end{equation}
and satisfies
\[
\sup_{Q_{\rho}} |u-P| \leq \rho^{2+\alpha}.
\]
\end{lemma}

\begin{proof}
We argue by contradiction. Suppose that the conclusion of the lemma is not true. Then, for all $\rho_0$ and sequences $(F_k)_{k\in \mathbb{N}}$, $(u_k)_{k\in \mathbb{N}}$, $(f_k)_{k\in \mathbb{N}}$, $(\Lambda_k)_{k\in \mathbb{N}}$ satisfying
\begin{equation}\label{eq_100}
   \|u_k\|_{L^\infty(Q_1)}\leq 1,
\end{equation}

\begin{equation}\label{eq_102}
\Theta_{F_k}(x, t) + \|f_k\|_{L^{\infty}(Q_{1})}  \leq \frac{1}{k}, \;\; \mbox{ and } \;\; \Lambda_k \leq \left(1+\frac{1}{k}\right)\lambda
\end{equation}

\begin{equation}\label{eq_103.5}
F_k \;\;\mbox{ is a uniformly } \; (\lambda, \Lambda_k)-\mbox{ elliptic operator},
\end{equation}
$u_k$ solves
\begin{equation}\label{eq_101}
\partial_t u_k-F_k(x, t, D^2u_k)= f_k(x, t) \;\; \mbox{ in } \; Q_1,
\end{equation}
in the viscosity sense, but
\begin{equation}\label{eq_103}
   \sup_{Q_{\rho_0}}|u_k-P| > \rho_0^{2+\alpha},
\end{equation}
for all quadratic polynomials $P$ that solve \eqref{eq_poly}.

From \eqref{eq_101}, we have that solutions $u_k$ are locally of class $C^{\beta, \beta/2}$, for some $\beta \in (0, 1)$, see for instance \cite[Section 5]{CKS00} and \cite[Section 4.4]{Wang92-II}. Hence, there exists a continuous function $u_\infty$ such that $u_k \to u_\infty$ (up to a subsequence) locally uniformly in $Q_1$. From \eqref{eq_102}, we have that $\Lambda_k \leq 2\lambda$ for all $k$ sufficiently large, which implies that the operators $F_k$ are $(\lambda, 2\lambda)$-elliptic for $k$ sufficiently large. Moreover, by using \eqref{eq_102}, we can conclude that $\Lambda \to \lambda$, up to a subsequence. Hence, $F_k \to \lambda\tr$ locally uniformly on ${Sym}(n)$ and by stability results (see \cite[Section 6]{CKS00} and \cite[Lemma 1.4]{Wang92-II}), we obtain

\begin{equation}\label{eq_104}
\partial_t u_\infty - \lambda\tr(D^2u_\infty) = 0 \;\;\mbox{ in }\; Q_{3/4}.
\end{equation}

Now, the classical regularity results available for \eqref{eq_104}, implies that $u_\infty$ is a smooth function and then we can define its Taylor polynomial
\[
P(x,t) = \dfrac{1}{2}x^TD^2u_\infty(0,0)x + \partial_t u_\infty(0,0)\cdot t + Du_\infty(0,0)\cdot x + u_\infty(0,0).
\]
In particular, from Lemma \ref{EstIntDer} we have for $\rho \in (0, 1/3)$,
\[
\sup_{Q_{\rho}}|u_\infty - P| \leq \mathrm{C}\rho^3,
\]
where $C$ is a universal constant. By using \eqref{eq_104} once more, we can infer that
\[
\partial_tP - \lambda\tr(D^2P) = 0 \;\; \mbox{ in } \; Q_{3/4},
\]
which implies
\[
\big|\partial_t P - F_k(0, 0, D^2P)\big| = o(1).
\]
Now, we set $a_k = \partial_t P - F_k(0, 0, D^2P)$ so that $|a_k| = o(1)$. In addition, we define the polynomial
\[
P_k := P - a_kt.
\]
Notice that $P_k$ solves
\[
\partial_t P_k - F_k(0,0,D^2P_k) = \partial_t P - F_k(0,0,D^2P)   - a_k = 0.
\]
Finally, combining the estimates above we conclude that for $k$ sufficiently large, and $\rho_0$ small
\begin{align*}
\sup_{Q_{\rho_0}}|u_k - P_k| & \leq \sup_{Q_{\rho_0}}|u_k - u_\infty| + \sup_{Q_{\rho_0}}|u_\infty - P| + \sup_{Q_{\rho_0}}|P - P_k| \\
& \leq \frac{1}{4}\rho_0^{2+\alpha} + \mathrm{C}\rho_0^3 + \frac{1}{4}\rho_0^{2+\alpha}.
\end{align*}
We conclude the prove by taking $\rho_0 \leq \left(\frac{1}{4C} \right)^{\frac{1}{1-\alpha}}$, so that
\[
\sup_{Q_{\rho_0}}|u_k - P_k| \leq \frac{3}{4}\rho_0^{2+\alpha},
\]
which is a contradiction with \eqref{eq_103}.
\end{proof}

\begin{proposition}\label{prop_interate1} Let $u \in C(Q_1)$ be a normalized viscosity solution to \eqref{eq_main}. Suppose that A\ref{Assump1}-A\ref{A3} are in force. If
\[
   \Theta_F(x, t) + \|f\|_{L^{\infty }(Q_{1})} \leq \varepsilon, \;\; \mbox{ and } \;\; \Lambda \leq (1+\varepsilon)\lambda,	\\\\\\
\]
then, there exists a sequence of second order polynomials $P_k$ of the form
\[
P_k(x,t) := \dfrac{1}{2}x^TD_kx + C_kt + B_k\cdot x + A_k,
\]
satisfying
\begin{equation}\label{eq_seq_inte}
\sup_{Q_{\rho^k}}|u(x,t) - P_k(x,t)| \leq \rho^{(2+\alpha)k},
\end{equation}
and
\begin{equation}\label{eq_condf}
C_k = F(0, 0, D_k)
\end{equation}
for every $k \geq 1$. Moreover,
\begin{equation}\label{eq_coeff_seq1}
|A_k-A_{k-1}| + \rho^{k-1}|B_k-B_{k-1}| + \rho^{k-1}(|C_k-C_{k-1}| + |D_k-D_{k-1}|) \leq C\rho^{(k-1)(2+\alpha)},
\end{equation}
where $C > 0$ is a universal constant and $\rho$ and $\varepsilon$ are as in Proposition \ref{lemma_app}.
\end{proposition}

\begin{proof}
We argue by an induction argument. Set $P_{-1}=P_0=0$, then, the case $k=1$ follows trivially. Now, suppose that the result holds true for $k=1, \ldots, m$, and lets prove it for $k=m+1$. We introduce the auxiliary function $v : Q_1 \to \R$ defined by
\[
v(x) = \dfrac{u(\rho^mx, \rho^{2m}t) - P_m(\rho^mx, \rho^{2m}t)}{\rho^{(2+\alpha)m}}.
\]
Notice that $v$ solves
\[
\partial_t v - G(x, t, D^2v) = \tilde{f}(x,t) \;\;\mbox{ in }\; Q_1,
\]
where
\[
G(x,t, M) := \rho^{-m\alpha}\left(F(\rho^{m}x, \rho^{2m}t, \rho^{m\alpha}M + D_m) -C_m\right)
\]
and
\[
\tilde{f}(x, t) = \rho^{-m\alpha}f(\rho^mx, \rho^{2m}t),
\]
Now, observe that (by definition and \eqref{eq_condf}) $G$ satisfies $A\ref{Assump1}$-$A\ref{A3}$, and
\[
\|\tilde{f}\|_{L^\infty(Q_1)} \leq \rho^{-m\alpha}\|f\|_{L^\infty(Q_{\rho^m})} \leq \varepsilon,
\]
where in the last inequality we have used the H\"{o}lder continuity of $f$. Therefore, we can apply Proposition \ref{lemma_app} to $v$ and find a polynomial $\tilde{P}$ of the form
\[
\tilde{P}(x,t) := \dfrac{1}{2}x^T\tilde{D}x + \tilde{C}t + \tilde{B}\cdot x + \tilde{A},
\]
with $|\tilde{A}|, |\tilde{B}|, |\tilde{C}|, |\tilde{D}| \leq M$, such that
\[
\sup_{Q_\rho}|v(x,t) - \tilde{P}(x,t)| \leq \rho^{2+\alpha},
\]
and
\[
\partial_t\tilde{P}-G(0,0, D^2\tilde{P}) = 0 \;\; \mbox{ in } \; Q_{3/4}.
\]
Rescaling back to the function $u$, we get
\[
\sup_{Q_{\rho^{m+1}}}|u(x,t) - P_{m+1}(x,t)| \leq \rho^{(2+\alpha)(m+1)},
\]
where $P_{m+1}(x,t) := P_m(x, t) + \rho^{m(2+\alpha)}\tilde{P}(\rho^{-m}x, \rho^{-2m}t)$. From the definition of $P_{m+1}$ we have
\[
C_{m+1} = F(0,0, D_{m+1}),
\]
and
\[
|A_{m+1}-A_{m}| + \rho^{m}|B_{m+1}-B_{m}| + \rho^{m}(|C_{m+1}-C_{m}| + |D_{m+1}-D_{m}|) \leq C\rho^{m(2+\alpha)}.
\]
This finishes the proof.
\end{proof}

Now, the proof of Theorem \ref{theo:main1} follows from standards arguments, that we will include here for sake of completeness.

\begin{proof}[Proof the Theorem \ref{theo:main1} ]
First, notice that the coefficients of the polynomials from Proposition \ref{prop_interate1} form a Cauchy sequence. Hence, we can infer the existence of a polynomial $\bar{P}$ of the form
\[
\bar{P}(x,t) := \dfrac{1}{2}x^T\bar{D}x + \bar{C}t + \bar{B}\cdot x + \bar{A},
\]
so that $P_k \to \bar{P}$ uniformly in $Q_1$. From \eqref{eq_coeff_seq1}, we have the estimates
\[
\left\{
\begin{aligned}
|\bar{A} - A_k| \leq C\rho^{k(2+\alpha)}, \;&\; |\bar{B} - B_k| \leq C\rho^{k(1+\alpha)} \\
|\bar{C} - C_k| \leq C\rho^{k\alpha} \; \mbox{ and} \;&\; |\bar{D} - D_k| \leq C\rho^{k\alpha}
\end{aligned}
\right.
\]
Finally, given $0 <r \ll 1/2$, let $k \in \N$ be such that $\rho^{k+1} \leq r \leq \rho^{k}$. We estimate,
\begin{align*}
\sup_{Q_r}|u(x,t) - \bar{P}(x, t)| & \leq \sup_{Q_{\rho^k}}|u(x,t) - P_k(x, t)| + \sup_{Q_{\rho^k}}|P(x,t) - \bar{P}(x, t)| \\
& \leq \frac{\mathrm{C}}{\rho}\rho^{(k+1)(2+\alpha)} \\
& \leq \mathrm{C}r^{2+\alpha}.
\end{align*}
This finishes the proof  by using the equivalence of semi-norms from Lemma \ref{Lem02}.

\end{proof}

\section{Parabolic Log-Lipschitz type estimates: Proof of Theorem \ref{theo:LL}}\label{Sect. Log-Lip}

In this intermediate Section, we shall comment on the $C^{1, \text{Log-Lip}}$ parabolic interior estimate, which will be obtained by fine adjustments in the arguments carried out in Section \ref{sect_scha}. For this purpose, we now assume that the source term $f$ satisfies $A\ref{ALL}$, i.e, that $f \in  L^\infty(Q_1)$. Notice that, it is possible to find a solution $u$ to
\[
\partial_t u(x, t)- \tr(\mathfrak{A}(x,t)D^2u) = f(x, t) \;\; \mbox{ in } \;\; Q_1,
\]
where $\mathfrak{A}$ and $f$ are continuous, but neither $\partial_t u$ nor $D_{ij}u$ are bounded. We refer the reader to \cite{Il'in67} and \cite{Kru67}.

\begin{proof}[Proof of Theorem \ref{theo:LL}]
We revisit the proof of Proposition \ref{prop_interate1}. As before, under the standard smallness assumptions, we can apply Lemma \ref{lemma_app} and find a universal $0<\rho<1/2$ and a quadratic polynomial $P_1$ such that
\begin{equation}\label{eqEst1LL}
\sup_{Q_{\rho}} |u-P_1| \leq \rho^2.
\end{equation}
Again, it is enough to find a sequence of quadratic polynomials
\[
P_k(x, t) := \frac{1}{2}x^T\cdot D_k \cdot x + C_kt + B_k\cdot x + A_k,
\]
such that
\begin{equation}\label{LLhipind}
C_k = F(0, 0, D_k) \;\; \mbox{ and } \;\; \sup_{Q_{\rho^k}} |u-P_k| \leq \rho^{2k}.
\end{equation}

We resort to an induction argument on $k$. The first step of induction, $k = 1$, follows from \eqref{eqEst1LL}. Now, suppose that the induction hypotheses have been established for $k = 1,\ldots, n$. Let us prove the case $k=n+1$. By defining the re-scaled function $v := Q_1 \to \R$ as
\[
v_n(x,t) = \frac{u(\rho^n x, \rho^{2n}t)-P_n(\rho^n x, \rho^{2n}t)}{\rho^{2n}},
\]
we have, by the induction hypotheses, that $\|v_n\|_{L^{\infty}(Q_1)} \leq 1$ and it solves in the viscosity sense,
\[
\partial_t v_n - F_n(x, t, D^2v_n) = f(\rho^n x, \rho^{2n}t) := f_n(x, t),
\]
where $F_n(x, t, M) := F(\rho x, \rho^2 t, M + D_n) - C_n$. Now, since $v_n$, $F_n$ and $f_n$ also satisfies the hypotheses of Lemma \ref{lemma_app}, we obtain another quadratic polynomial $P$ satisfying
\begin{equation}\label{eq6.5}
\sup_{Q_{\rho}} |v_n - P| \leq \rho^2.
\end{equation}
Hence, rescaling back to $u$, we can infer
\begin{equation}\label{eq6.6}
\sup_{Q_{\rho^{n+1}}} \left|u(x,t) - \left[P_n(x,t)+ \rho^{2k}P\left(\frac{x}{\rho^n}, \frac{t}{\rho^{2n}}\right)\right]\right|\leq \rho^{2(n+1)}.
\end{equation}
Therefore, by defining
\[
P_{n+1}(x,t) := P_n(x,t) + \rho^{2n}P\left(\frac{x}{\rho^n}, \frac{t}{\rho^{2n}}\right),
\]
we check the $(n+1)^{{th}}$-step of induction and, the desired  conditions in \eqref{LLhipind} are satisfied. Moreover, we have the following approximation rate
\begin{equation}\label{eqLLcoeff}
|A_k-A_{k-1}| + \rho^{k-1}|B_k-B_{k-1}| + \rho^{2(k-1)}(|C_k-C_{k-1}| + |D_k-D_{k-1}|) \leq C\rho^{2(k-1)},
\end{equation}
for every $k \geq 1$, which implies
\begin{equation*}
|A_k-u(0, 0)| \leq C\rho^{2k} \;\; \mbox{ and } \;\;  |B_k- Du(0, 0)| \leq C\rho^{k};
\end{equation*}
as well as,
\begin{equation*}
|D_k| \leq Ck \;\; \mbox{ and } \;\;  |C_k| \leq Ck.
\end{equation*}

Finally, from the estimates above we are able to conclude that
\begin{equation}\label{eq_LL}
    \sup_{Q_r} \left | u(x, t) - \left [ u(0, 0) + D u(0, 0) \cdot x \right ] \right | \le \mathrm{C} r^2 \log r^{-1},
\end{equation}
for a constant $\mathrm{C}>0$ that depends only upon $d,\lambda$ and $\Lambda$. This concludes the proof of Theorem \ref{theo:LL}.
\end{proof}

\section{Estimates in Sobolev spaces}

In this section, we aim to detail the proof of Theorem \ref{theo:main3}. The proof follows standard arguments once we have Lemma \ref{lemma_app} available, and it consists in proving certain decay rates for the sets $A_M$. See for instance \cite{CastPim17} and \cite{ARS22}. Nevertheless, we will include the proof here on account of completeness. We start with a classic result, that gives us a first decay rate of the sets $A_M$.

\begin{proposition}[{\bf Estimate in $W_{loc}^{2, 1; \delta}(Q_1)$}] \label{prop_sob1}
Let $u\in C(Q_1)$ be a normalized viscosity solution to (\ref{eq_main}). Assume A\ref{Assump1}, A\ref{A2} and A\ref{A4} hold true. Then, there exist $\delta>0$ and a universal constant $C>0$ such that
\[
| A_M(u, Q_1)\cap K_1| \leq \mathrm{C} M^{-\delta}.
\]
\end{proposition}
\noindent For a proof of Proposition \ref{prop_sob1} we refer the reader to \cite{CabCaf2003} in the elliptic case and \cite{Wang92-I} in the parabolic case.

From now on, our goal is to accelerate the decay rate of the sets $A_M$, by making use of Lemma \ref{lemma_app}. The first step in this direction is to combine Lemma \ref{lemma_app} and Proposition \ref{prop_sob1} to obtain a first level of refined decay rate. Let $Q$ be a parabolic domain such that $Q_{8\sqrt{d}} \subset Q$.

\begin{proposition}\label{prop_sob2}
Let $\rho \in (0,1)$ and $u\in C(Q)$ be a normalized viscosity solution to \eqref{eq_main} in $Q_{8{\sqrt{d}}}$ such that
\[
- |x|^2 - |t| \leq u(x,t) \leq |x|^2 + |t| \;\; \mbox{ in } \;\; Q \setminus Q_{6{\sqrt{d}}}.
\]
Assume that A\ref{Assump1}, A\ref{A2} and A\ref{A4} are in force and that
\[
\|f\|_{L^{d+1}(Q_{8\sqrt{d}})} \leq \varepsilon.
\]
Then, we can find $M > 1$ satisfying
\[
| G_{M}(u, Q) \cap K_1| \geq 1 - \rho.
\]
\end{proposition}

\begin{proof}
Let $h$ be the smooth function from Lemma \ref{lemma_app}. In particular we have $h\in {\mathcal C}^{1,1}_{loc}(Q_{8\sqrt{d}})$ and
\[
\|u-h\|_{L^{\infty}(Q_{6\sqrt{d}})} \leq \delta.
\]
Now, we extend $h$ continuously to $Q$, so that
\[
h = u \;\; \mbox{ in } \;\; Q\setminus Q_{7\sqrt{d}}
\]
with
\[
\|u-h\|_{L^\infty(Q)} = \|u-h\|_{L^\infty(Q_{6\sqrt{d}})}.
\]
From the maximum principle, we can infer
\[
\|u\|_{L^\infty(Q_{6\sqrt{d}})} = \|h\|_{L^\infty(Q_{6\sqrt{d}})}.
\]
Hence,
\[
\|u-h\|_{L^\infty(Q)} \leq 2,
\]
and
\[
-2-|x|^2 - |t| \leq h(x,t) \leq 2  + |x|^2 + |t| \;\; \mbox{ in } \;\; Q \setminus Q_{6\sqrt{d}}.
\]
It follows that, there exists $N > 1$ such that $Q_1\subset G_{N}(h, Q)$.

We define the auxiliary function $w: Q_{8\sqrt{d}} \to \R$ as
\[
 w := \frac{\delta}{2\mathrm{C} \varepsilon}(u - h).
 \]
Notice that $w \in S(1, \Lambda, f)$, thus, we can apply Proposition \ref{prop_sob1} to $w$ and conclude that
\[
|A_{{M}_1}(w, Q) \cap K_1| \leq \mathrm{C} {M}_1^{-\sigma},
\]
for every ${M}_1>0$, which implies
\[
|A_{{M}_2}(u-h , Q) \cap K_1| \leq \mathrm{C} {\varepsilon}^{\sigma}{M}_2^{\sigma},
\]
for every ${M}_2>0$. Therefore
\[
|G_{N}(u-h, Q) \cap K_1| \geq 1 - \mathrm{C}{\varepsilon}^{\sigma}{M}_2.
\]
By choosing $\varepsilon$ sufficiently small and ${M} \equiv 2N$, we conclude the proof.
\end{proof}

\begin{proposition}\label{prop_sob3}
Let $\rho \in (0,1)$ and $u\in C(Q)$ be a normalized viscosity solution to \eqref{eq_main} in $Q_{8\sqrt{d}}$. Assume A\ref{Assump1}, A\ref{A2} and A\ref{A4} are in force. In addition, suppose
\[
\|f\|_{L^{d+1}(Q_{8\sqrt{d}})} \leq \varepsilon,
\]
and $G_1(u,Q)\cap K_3\not=\emptyset.$ Then
\[
|G_{M}(u, Q) \cap K_1| \geq 1 - \rho,
\]
with ${M}$ as in Proposition \ref{prop_sob2}.
\end{proposition}

\begin{proof}
Let $(x_1, t_1) \in G_1(u,Q)\cap K_3$. Then, by definition, we can find an affine function $L$ satisfying
\[
-\dfrac{|x - x_1|^2 + |t - t_1|}{2} \leq u(x,t) - L(x,t) \leq \dfrac{|x - x_1|^2 + |t - t_1|}{2} \;\; \mbox{ in } \; Q.
\]
Now, we introduce the auxiliary function
\[
v:= \dfrac{u-L}{\mathrm{C}},
\]
where $C>1$ is a sufficiently large constant so that $\|v\|_{L^{\infty}(Q_{8\sqrt{d}})} \leq 1$ and
\[
-|x|^2-|t| \leq v(x,t) \leq |x|^2 + |t| \;\; \mbox{ in } \; Q\setminus Q_{6\sqrt{d}}.
\]
We have that $v$ solves
\[
v_t - \dfrac{1}{\mathrm{C}}F(x, t, \mathrm{C}D^2u) = \dfrac{f}{\mathrm{C}},
\]
in the viscosity sense. By setting ${M}:=C\bar{{M}}$, Proposition \ref{prop_sob2} infers that
\[
|G_{M}(u,Q)\cap K_1| = |G_{C\bar{{M}}}(u,Q)\cap K_1| = |G_{\bar{{M}}}(v,Q)\cap K_1| \geq 1 - \rho.
\]
\end{proof}

In the next, we apply Lemma \ref{lem_cov}, which yields to an improved decay rate for the sets $A_M \cap K_1$.

\begin{proposition}\label{prop_sob4}
Let $\rho \in (0,1)$ and $u \in C(Q)$ be a normalized viscosity solution to \eqref{eq_main} in $Q_{8\sqrt{d}}$. Extend $f$ by zero outside of $Q_{8\sqrt{d}}$. Suppose A\ref{Assump1}, A\ref{A2} and A\ref{A4} are satisfied and set
\[
A := A_{{M}^{k+1}}(u,Q_{8\sqrt{d}})\cap K_1,
\]
and
\[
B := \left\{ A_{{M}^{k}}(u, Q_{8\sqrt{d}})\cap K_1\right\}\cup\left\{(x,t)\in K_1| m(f^{d+1})(x,t)\geq(C{M}^{k})^{d+1}\right\},
\]
where ${M}>1$ depends only on $d$ and $C >0$ is a universal constant and . Then,
\[
|A| \leq \rho|B|.
\]
\end{proposition}

\begin{proof}
Notice that
\[
|u(x,t)| \leq 1 \leq |x|^2 + |t| \;\; \mbox{ in } \; Q_{8\sqrt{d}}\setminus Q_{6\sqrt{d}}.
\]
Therefore from \ref{prop_sob2}, we get
\[
|G_{{M}^{k+1}}(u, Q_{8\sqrt{d}})\cap K_1| \geq 1 - \rho,
\]
or equivalently,
\[
|A| = |A_{{M}^{k+1}}(u, Q_{8\sqrt{d}})\cap K_1| \leq \rho|K_1|.
\]
Let $K := K_{1/2^i}$ be an arbitrary dyadic cube of $K_1$. We have
\begin{equation}\label{eq_hyp}
|A_{{M}^{k+1}}(u,Q_{8\sqrt{d}})\cap K|=|A\cap K| > \rho|K|.
\end{equation}
Now, in order to apply Lemma \ref{lem_cov}, it remains to show that $\bar{K}^m \subset B$, for some $m \in \mathbb{N}$. We argue by contradiction. Suppose that for all $m \in \mathbb{N}$ we have $\bar{K}^m \not\subset B$. Let $(x_1,t_1)$ be such that for any $m \in \mathbb{N}$,
\begin{equation}\label{eq_cont}
(x_1,t_1) \in \bar{K}^m\cap G_{{M}^k}(u, Q_{8\sqrt{d}})
\end{equation}
and
\begin{equation}\label{eq_max_con}
m(f^{d+1})(x_1,t_1) \leq (C_1{M}^k)^{d+1}.
\end{equation}
We introduce the auxiliary function
\[
v(x,t):= \frac{2^{2i}}{{M}^k}u\left(\frac{x}{2^i},\frac{t}{2^{2i}}\right).
\]
Observe that $Q_{8\sqrt{d}} \subset Q_{2^i\cdot8\sqrt{d}}$. Hence, $v$ solves in the viscosity sense
\[
\partial_t v - \frac{1}{{M}^k}F(x,t, {M}^kD^2v) = \tilde{f} \;\; \mbox{ in } \; Q_{8\sqrt{d}},
\]
where
\[
\tilde{f}(x,t) := \frac{1}{{M}^k}f\left(\frac{x}{2^i},\frac{t}{2^{2i}}\right).
\]
Now, we estimate the $L^{d+1}$-norm of $\tilde{f}$:
\[
\|\tilde{f}\|_{L^{d+1}(Q_{8\sqrt{d}})}^{d+1} = \frac{2^{i(d+2)}}{{M}^{k(d+1)}}\int_{Q_{8\sqrt{d}/2^i}}|f(x,t)|^{d+1}dxdt\leq c(d)C_1^{d+1},
\]
and by taking $C_1$ sufficiently small in \eqref{eq_max_con}, we obtain that $\tilde{f}$ satisfies
\[
\|\tilde{f}\|_{L^{d+1}(Q_{8\sqrt{d}})} \leq \varepsilon.
\]
Moreover, \eqref{eq_cont} implies
\[
G_1(v,Q_{8\sqrt{d}/2^i})\cap K_3 \neq\emptyset.
\]
Therefore, Proposition \ref{prop_sob3} yields to
\[
|G_{M}(v, Q_{2^i \cdot 8\sqrt{d}})\cap K_1| \geq (1-\rho),
\]
and rescaling back to $u$ we conclude
\[
|G_{{M}^{k+1}}(u, Q_{8\sqrt{d}})\cap K| \geq (1-\rho)|K|,
\]
which is a contradiction with \eqref{eq_hyp}. This finishes the proof.
\end{proof}

We now have gathered all we need to prove Theorem \ref{theo:main3}.

\begin{proof}[Proof of Theorem \ref{theo:main3}]
We start by defining the quantities
\[
\alpha_k := |A_{M^k}(u,Q_{8\sqrt{d}})\cap K_1|
\]
and
\[
\beta_k := |\{(x,t)\in K_1:m(f^{d+1})(x,t)\geq(C_1{M}^k)^{d+1}\}|.
\]
By applying Proposition \ref{prop_sob4}, we can infer that
\[
\alpha_{k+1} \leq \rho(\alpha_k + \beta_k),
\]
which implies
\begin{equation}\label{eq_sob10}
\alpha_k\leq\rho^k+	\displaystyle\sum_{i=1}^{k-1}\rho^{k-i}\beta_i.
\end{equation}

From the fact that $f\in L^p(Q_1)$, we obtain $m(f^{d+1})\in L^{p/(d+1)}(Q_1),$ with the estimate
\[
\|m(f^{d+1})\|_{L^{p/(d+1)}(Q_1)}\leq \mathrm{C}\|f\|^{d+1}_{L^p(Q_1)},
\]
for some positive constant $C$. Therefore, we get
\begin{equation}\label{eq_sob11}
\displaystyle\sum_{k=0}^{\infty}M^{pk}\beta_k\leq \mathrm{C}.
\end{equation}

Finally, by putting \eqref{eq_sob10} and \eqref{eq_sob11} together, and taking $\rho$ so that $\rho{M}^p \leq 1/2$, we conclude
\begin{align*}
\sum_{k=1}^{\infty}{M}^{pk}\alpha_k & \leq \sum_{k=1}^{\infty}(\rho {M}^p)^k + \sum_{k=1}^{\infty}\sum_{i=0}^{k-1}\rho^{k-i}{M}^{p(k-i)}\beta_i{M}^{pi} \\
& \leq \sum_{k=1}^{\infty}2^{-k}+\left(\sum_{i=0}^{\infty}{M}^{pi}\beta_i\right)\left(\sum_{j=1}^{\infty}(\rho {M}^p)^j\right)\\
& \leq \sum_{k=1}^{\infty}2^{-k} + \mathrm{C}\sum_{j=1}^{\infty}2^{-j}\\
&\leq \mathrm{C},
\end{align*}
and the result follows.
\end{proof}

\section{Some consequences and final comments}

In this section, we give some applications to our findings, comparing them with some results in the literature.

\subsection{$p-$BMO type estimates}

We start with $p$-BMO type estimates, $p \in (1, \infty)$, for $\partial_t u$ and $D^2 u$, for the solutions of \eqref{eq_main}. Notice that the final estimate in \eqref{eq_LL} indicates that solutions to \eqref{eq_main} exhibit asymptotic $C^{2, 1}$ behaviour in the parabolic context. Moreover, both $\partial_t u$ and $D^2 u$ display logarithmic tendencies near the origin. Hence, we can obtain $p$-BMO estimates for $\partial_t u$ and $D^2 u$ by adapting the previous proofs. We first give the definition of $p$-BMO norm.

\begin{definition}
We recall that a function $f\in L^{1}_{loc}(Q)$ is said of \textbf{$p$-bounded mean oscillation} in $Q$ for $p\in[1,\infty)$ or $f\in p-BMO(Q)$ if

\begin{equation*}\label{p-BMOnorm}
\|f\|_{p-BMO(\mathrm{Q})} := \sup_{(x_{0}, t_0)\in \mathrm{Q}, \rho>0} \left(\intav{\mathrm{Q}_{\rho}(x_0, t_0) \cap \mathrm{Q}} |f(x, t) - (f)_{(x_0,t_0), \rho}|^p dx\right)^{\frac{1}{p}} <\infty,
\end{equation*}
where for each $(x_{0}, t_0)\in \mathrm{Q}$ and $\rho>0$ we have that
\[
(f)_{(x_0,t_0), \rho} := \intav{\mathrm{Q}_{\rho}(x_0, t_0) \cap \mathrm{Q}} f(x, t) dxdt
\]
Moreover, for sake of simplicity, we denote $(f)_{\rho}$ when $(x_{0}, t_0)=(0, 0)$.
\end{definition}

With this definition in hand, we have the following result:

\begin{proposition}
Under the hypotheses of Theorem \ref{theo:LL}, we have
\[
\left\|\partial_t u\right\|_{p-BMO(Q_r)}+ \|D^2 u\|_{p-BMO(Q_r)} \leq \mathrm{C}\left(\|u\|_{L^{\infty}(Q_1)} + \|f\|_{L^{\infty}(Q_1)}\right),
\]
where $\mathrm{C} > 0$ only depends on $d$, $p$, $\lambda$ and $\Lambda$.
\end{proposition}

\begin{proof}
We argue as in the proof of Theorem \ref{theo:LL}, under the corresponding smallness regime on $\Theta_F$ and $f  \in L^{\infty}(Q_1)$. As before, we can find a sequence of quadratic polynomials
\[
P_k(x, t) := \frac{1}{2}x^T\cdot D_k \cdot x + C_k.t+ B_k\cdot x+A_k,
\]
for which the auxiliary functions
\[
v_k(x,t) := \dfrac{u(\rho^kx, \rho^{2k}t)-P_k(\rho^kx, \rho^{2k}t)}{\rho^{2k}},
\]
are such that
\[
\|v_k\|_{L^\infty(Q_1)} \leq 1,
\]
and
\[
\partial_t v_k - F_k(x,t, D^2v_k) = f_k(x,t) \;\; \mbox{ in } \;\; Q_1.
\]
Now, we observe that, $F_k$ and $f_k$ satisfies assumptions $A\ref{Assump1}, A\ref{A2}$ and $A\ref{A4}$, and hence we can apply Theorem \ref{theo:main3} to conclude
\[
\left\|\partial_t v_k\right\|_{L^p\left(Q_{1/2}\right)}, \|D^2 v_k\|_{L^p\left(Q_{1/2}\right)} \leq \mathrm{C}.
\]
Hence,
\[
\intav{Q_{\rho^k/2}}\left|\partial_t(u-P_k)\right|^p + |D^2(u-P_k)|^p \leq \mathrm{C},
\]
for all $k \in \N$. Equivalently, we have
\[
\left\|\partial_t u\right\|_{p-BMO(Q_r)}+ \|D^2 u\|_{p-BMO(Q_r)} \leq \mathrm{C}(n, p, \lambda, \Lambda)\left(\|u\|_{L^{\infty}(Q_1)} + \|f\|_{L^{\infty}(Q_1)}\right),
\]
for all $r \ll 1$. This finishes the proof,
\end{proof}

\subsection{Estimates for a class of solutions in fully nonlinear models}

As a direct consequence of our findings, we recover and improve the recent results from Lee-Yu in \cite{LeeYu23}. In fact, let $u \in C(Q_1)$ be such that the following inequalities hold in the viscosity sense:
\begin{equation}\label{eq_class}
\partial_t u - \mathcal{M}^{+}_{\lambda,\Lambda}(D^2 u) -\|f\|_{L^{\infty}(Q_1)} \leq 0 \leq \partial_t u - \mathcal{M}^{-}_{\lambda,\Lambda}(D^2 u) +\|f\|_{L^{\infty}(Q_1)},
\end{equation}
where $f \in L^{\infty}(Q_1)$. Whenever a function $u \in C(Q_1)$ satisfies \eqref{eq_class}, we say that $u \in \mathcal{S}^{\ast}(\lambda, \Lambda, f)$. In this scenario, we have the following result:

\begin{theorem}[{\bf \cite[Theorem 1.1]{LeeYu23}}] Let $\alpha \in (0, 1)$ and $u \in C(Q_1)$ satisfying \eqref{eq_class}. There exists $\delta >0$ depending only on $d$ and $\alpha$ such that if $\frac{\Lambda}{\lambda}-1 \le \delta$, then $u \in C_{\text{loc}}^{1+\alpha, \frac{1+\alpha}{2}}(\mathrm{Q}_1)$ and there exist a positive constant $C = C(d, \lambda, \alpha)$ such that
\[
\|u\|_{C^{1+\alpha, \frac{1+\alpha}{2}}(Q_{1/2})} \leq \mathrm{C}\left(\|u\|_{L^{\infty}(Q_1)} + \|f\|_{L^{\infty}(Q_1)}\right).
\]
\end{theorem}

Recall that if $u \in C(Q_1)$ is a viscosity solution to \eqref{eq_main} for some $(\lambda, \Lambda)$-elliptic fully nonlinear operator $F$, then $u \in \mathcal{S}(\lambda, \Lambda, f(x, t)+F(x, t, 0))$, i.e, $u$ satisfies
\[
\partial_t u - \mathcal{M}^{+}_{\lambda,\Lambda}(D^2 u)  \leq f(x, t) + F(x, t, 0) \leq \partial_t u - \mathcal{M}^{-}_{\lambda,\Lambda}(D^2 u) \;\; \mbox{ in } \;\; Q_1.
\]
Therefore, by applying Theorem \ref{theo:LL}, we obtain the following estimates
\[
\sup_{Q_r} \left | u(x, t) - \left [ u(0, 0) + D u(0, 0) \cdot x \right ] \right | \leq \mathrm{C} r^2 \log r^{-1}\left(\|u\|_{L^{\infty}(Q_1)} + \|f + F(x, t, 0)\|_{L^{\infty}(Q_1)} \right).
\]
In particular, given any $\alpha \in (0, 1)$, we have that $u$ is of class $C^{1+\alpha, \frac{1+\alpha}{2}}$ at $(0, 0)$ with the estimate
\[
\|u\|_{C^{1+\alpha, \frac{1+\alpha}{2}}(0, 0)} \leq \mathrm{C}\left(\|u\|_{L^{\infty}(Q_1)} + \|f + F(x, t, 0)\|_{L^{\infty}(Q_1)} \right)
\]
Finally, we must highlight that our estimates can be applied for a class of evolution PDEs driven by the Normalized $p$-Laplacian for $p \in (1, \infty)$, i.e.,
$$
\partial_t u(x, t) - \Delta^{\mathrm{N}}_p u(x, t) = f \in L^{\infty}(\mathrm{Q}_1) \quad \text{with} \quad \frac{\Lambda_p(\Delta^{\mathrm{N}}_p)}{\lambda_p(\Delta^{\mathrm{N}}_p)} = \frac{\min\{1, p-1\}}{\max\{1, p-1\}} \to 1 \quad \text{as} \quad p \to 2.
$$
Specifically, our estimates extend, in some extend, the ones addressed by Andrade-Santos in \cite{AndSant22} for $|p-2| = \text{o}(1)$.

\subsection{Estimates for a class of fully nonlinear Isaac operators}

Of particular interest, Theorem \ref{theo:main1} covers {\it Isaac's type equations}, which appear in Stochastic Control and in the Theory Differential Games:
 \begin{equation} \label{Isaac}
	\displaystyle \partial_t  u - F(x, t,  D^2u) = \partial_t  u - \sup_{\beta\in \mathcal{B}} \inf_{\gamma \in \mathcal{A}} (L_{\gamma \beta} u(x, t) - f_{\gamma \beta}(x, t)) = 0 \quad \text{in} \quad \mathrm{Q}_1,
\end{equation}
where $f_{\gamma \beta }\colon Q_1 \rightarrow \mathbb{R}$ are H\"{o}lder continuous and $L_{\gamma \beta}u = a_{\gamma \beta}^{ij}(x, t)\partial _{ij}u$ is a family of elliptic operators with H\"{o}lder continuous coefficients and ellipticity constants $\lambda$ and $\Lambda$ satisfying $\frac{\Lambda}{\lambda} -1 < \varepsilon_0$, where $\varepsilon_0$ is given by Theorem \ref{theo:main1}.

In the sequel, we present a regularity estimate for Isaac operators under a suitable smallness condition on the coefficients. Compare with \cite[Theorem 2.1]{ARS22}, \cite[Theorem 2.2]{ARS22} and \cite[Theorem 2.3]{ARS22}.

\begin{corollary}\label{Corolario1}
Let $u \in C(Q_1)$ be a viscosity solution to \eqref{Isaac}, with
\begin{equation}\label{eq:smallelip}
\frac{\Lambda}{\lambda} -1 < \varepsilon_0.
\end{equation}
Then
\begin{itemize}
\item[i)] If $f_{\gamma \beta} \in C^{\alpha, \frac{\alpha}{2}}(Q_1)$ and $a_{\gamma \beta}^{ij} \in C^{\alpha, \frac{\alpha}{2}}(Q_1)$, for $\alpha \in (0,1)$ , then $u \in C_{\text{loc}}^{2+\alpha,\frac{2+\alpha}{2}}(Q_1)$. Furthermore, an estimate like \eqref{EstSchauder} holds true.
\item[ii)] If $f_{\gamma\beta} \in L^\infty(Q_1)$, then $u \in C^{1, Log-Lip}_{loc}(Q_1)$.
\item[iii)] If $f_{\gamma\beta} \in L^p(Q_1)$, $d+1 <p < \infty $, then $u \in W^{2,1,p}_{loc}(Q_1)$.
\end{itemize}
\end{corollary}

\begin{remark} Note that under the smallness condition \eqref{eq:smallelip}, we have that there exists $\varepsilon^{\prime}_0 = \varepsilon_0^{\prime}(\varepsilon_0)\ll 1$ such that
\[
\left\|\frac{a_{\gamma \beta}^{ij}}{\lambda}-\delta_{ij}\right\|_{L^{\infty}(Q_1)} <\varepsilon_0^{\prime},
\]
which resembles a type of Cordes-Niremberg condition. In fact, we have
\[
L_{\gamma \beta}u(x, t) = \sum_{i, j=1}^{d} a_{\gamma \beta}^{ij }(x, t)\partial _{ij}u(x, t)
\]
for a $(\lambda, \Lambda)$-uniformly elliptic and symmetric matrix $A(x, t) = (a_{\gamma \beta}^{ij }(x, t))_{i,j=1}^d$, then
\[
L_{\gamma \beta}u(x, t) = \Delta u(x, t) +  \sum_{i, j=1}^{d} (a_{\gamma \beta}^{ij }(x, t)- \delta_{ij})\partial _{ij}u(x, t).
\]

Now, let $\{\xi_1, \xi_2, \cdots, \xi_d\}$ be unit vectors in the canonical basis of $\mathbb{R}^d$. For $(x_0, t_0) \in Q_1$, from the uniform ellipticity of ${A}$ we obtain
\begin{equation}\label{EqElipUnif}
\lambda|\xi|^2\le \langle a_{\gamma \beta}^{ij}(x_0, t_0)\xi, \xi\rangle \le \Lambda|\xi|^2 \quad \forall\,\,\xi \in \mathbb{R}^n.
\end{equation}

By choosing $\xi = \xi_k$ ($1 \le k \le d$) in \eqref{EqElipUnif}, we get $\lambda \le a_{\gamma \beta}^{kk}(x_0, t_0)\le \Lambda$. Hence, we can conclude
\begin{equation}\label{Eq1Cor1}
  \left|\frac{a_{\gamma \beta}^{kk}(x_0, t_0)}{\lambda}-1\right|\le \frac{\Lambda}{\lambda}-1< \varepsilon_0.
\end{equation}

On the other hand, by taking $\xi = \xi_i + \xi_j$ in \eqref{EqElipUnif} we obtain,
\[
2\lambda \leq a_{\gamma \beta}^{ii}(x_0, t_0) + a_{\gamma \beta}^{jj}(x_0, t_0) + a_{\gamma \beta}^{ij}(x_0, t_0) + a_{\gamma \beta}^{ji}(x_0, t_0) \le 2\Lambda,
\]
which implies
\[
0\le \left(\frac{a_{\gamma \beta}^{ii}(x_0, t_0)}{\lambda}-1\right) + \left(\frac{a_{\gamma \beta}^{jj}(x_0, t_0)}{\lambda}-1\right) + 2\frac{a_{\gamma \beta}^{ij}(x_0, t_0)}{\lambda}\le 2\left(\frac{\Lambda}{\lambda}-1\right).
\]
The triangular inequality yields to
\begin{equation}\label{Eq2Cor1}
\begin{array}{rcl}
 \displaystyle \left|\frac{a_{\gamma \beta}^{ij}(x_0, t_0)}{\lambda}-0\right| & \le & \displaystyle \frac{1}{2}\left[2\left(\frac{\Lambda}{\lambda}-1\right) + \left|\frac{a_{\gamma \beta}^{ii}(x_0, t_0)}{\lambda}-1\right| + \left|\frac{a_{\gamma \beta}^{jj}(x_0, t_0)}{\lambda}-1\right|\right] \\
   & < & 2\varepsilon_0.
\end{array}
\end{equation}

Therefore, from \eqref{Eq1Cor1} and \eqref{Eq2Cor1} we conclude that
\[
\frac{\Lambda}{\lambda}-1 <\varepsilon_0 \quad \Rightarrow \quad \left\|\frac{a_{\gamma \beta}^{ij}}{\lambda}-\delta_{ij}\right\|_{L^{\infty}(Q_1)} < 2\varepsilon_0,
\]
i.e. $\varepsilon_0^{\prime} = 2\varepsilon_0$.
\end{remark}

\subsection{Final comments}

Returning to one of the key questions explored in this work, namely, the quest for a classical solution, it is also available in the literature a partial regularity result. Before presenting this result, let's revisit some of its foundational concepts. We begin by introducing the notion of the ``parabolic Hausdorff dimension" for a set $\Omega \subseteq \R^{d+1}$:
\[
\mathcal{H}_{par}(\Omega) := \inf\left\{0 \leq s < \infty: \forall \,\gamma > 0 \,\,\exists \, \{Q_{r_j}(x_j, t_j)\}_{j \geq 1} \,s.t. \,\, \Omega \subseteq \bigcup_{j \geq 1} Q_{r_j}(x_j, t_j) \, \mbox{ and } \, \sum_{j \geq 1} r^s_j < \gamma \right\}
\]
Additionally, the relationship between the parabolic Hausdorff measure and the standard Hausdorff one is given by
\[
2\mathcal{H}(\Omega)-d \leq \mathcal{H}_{par}(\Omega) \leq \mathcal{H}(\Omega)+1.
\]
Under the notations above, we have the following result:

\begin{theorem}[{\bf \cite{D} and \cite[Theorem 4]{DasdosP19}}]\label{partial reg} Let $u \in C(Q_1)$ be a viscosity solution to
\[
   \partial_t u  - F(D^2u) = f(x, t) \;\; \mbox{ in } \;\; Q_1,
\]
where $F \in C^1(\text{Sym}(d))$ satisfies $\mathfrak{c}  \le D_{ij} F(M) \leq \mathfrak{c}^{-1}$ for some constant $\mathfrak{c}>0$, and $f$ is a Lipschitz continuous function. Then, there exist $\varepsilon >0$, depending only on universal parameters, and a closed set $\Gamma_{\mbox{Sing}} \subset Q_1$, with $\mathcal{H}_{\textit{par}}(\Gamma_{\text{Sing}}) \leq d+2-\varepsilon$, such that $u \in C^{2+\alpha, \frac{2+\alpha}{2}}(Q_1 \setminus \Gamma_{\mbox{Sing}})$ for all $\alpha \in (0, 1)$.
\end{theorem}

The previous result raise the following question: considering viscosity solutions to \eqref{eq_main}, what can we say about the constant $\varepsilon$ that appear in Theorem \ref{partial reg}? Notice that under the hypotheses of Theorem \ref{theo:main1}, we also show that $\mathcal{H}_{\textit{par}}(\Gamma_{\text{Sing}}) = 0$. Hence, our findings hints that the quantity $\varepsilon$ has some relation with $\mathfrak{e} :=  \frac{\Lambda}{\lambda} - 1$. In the elliptic framework, this was confirmed by the work of Armstrong-Silvestre-Smart in \cite{ASS12}, and more recently in the work by Nascimento-Teixeira \cite{NasTei23} (just to cite a few), where they proved the estimate
\[
 \frac{\left(1+\frac{2}{3}\left(1-\frac{\lambda}{\Lambda}\right)^{d-1}\right)}{\ln(d^4)}.\left(\frac{\lambda}{\Lambda}\right)^{d-1} \leq  \varepsilon \leq \dfrac{d\lambda}{(d-1)\Lambda + \lambda} = \dfrac{d}{(d-1)\frac{\Lambda}{\lambda}+1} \quad \text{for} \quad d\ge 3.
\]
We highlight that such a constant $\varepsilon$ in Theorem  \ref{partial reg} depends only on the dimension and ellipticity parameters of $F$. Moreover, it also appears in the Daniel's work \cite[Theorem 1.2]{D}, where he addressed a parabolic $W^{3, \varepsilon}$ estimate.

Therefore, in this direction, for our class of operators, and $d \geq 2$, it is reasonable (and we conjecture) to expect that.
\[
\mathcal{H}(\Gamma_{\text{Sing}}) \le (d+1)\left(1 - \text{c}\left(d, \frac{\Lambda}{\lambda}\right)\right),
\]
where $\text{c}\left(d, \frac{\Lambda}{\lambda}\right) \approx 1$ as $\frac{\Lambda}{\lambda} \approx 1$. Furthermore, we conjecture $\text{c}\left(d, \frac{\Lambda}{\lambda}\right) \lesssim \dfrac{d+1}{d\left(\frac{\Lambda}{\lambda}\right)+1}$.

\subsection*{Acknowledgments}

J.V. da Silva has been partially supported by CNPq-Brazil under Grant No. 307131/2022-0 and FAEPEX-UNICAMP 2441/23 Editais Especiais - PIND - Projetos Individuais (03/2023). M. Santos was partially supported by the Portuguese government through FCT-Funda\c c\~{a}o para a Ci\^{e}ncia e a Tecnologia, I.P., under the projects UID/MAT/04459/2020 and PTDC/MAT-PUR/1788/2020. We would like to thank the Instituto de Matem\'{a}tica Pura e Aplicada, IMPA (Rio de Janeiro - Brazil) for fostering a pleasant and productive scientific atmosphere during the $34$th Brazilian Mathematics Colloquium, where part of this research work was developed.

\bigskip

\noindent\textsc{Makson S. Santos}\\
Departamento de Matem\'atica do Instituto Superior T\'ecnico\\
Universidade de Lisboa\\
1049-001 Lisboa, Portugal\\
\noindent\url{makson.santos@tecnico.ulisboa.pt}

\vspace{.15in}

\noindent\textsc{Jo\~{a}o Vitor da Silva}\\
Departamento de Matem\'{a}tica\\
Universidade Estadual de Campinas - UNICAMP, \\
13083-970, Barão Geraldo, Campinas - SP, Brazil\\
\noindent\url{jdasilva@unicamp.br}

\end{document}